\documentclass[leqno,12pt]{amsart}
\usepackage{amsmath}
\usepackage{amsfonts}
\usepackage{amssymb}
\usepackage{amsthm}
\usepackage{graphicx}
\usepackage[usenames,dvipsnames]{color}
\usepackage{hyperref}
\usepackage{subfigure}
\newtheorem{thm}{Theorem}
\newtheorem{cor}[thm]{Corollary}
\newtheorem{lem}[thm]{Lemma}
\newtheorem{prop}[thm]{Proposition}
\theoremstyle{definition}
\newtheorem{defn}[thm]{Definition}
\theoremstyle{remark}
\newtheorem{eg}[thm]{Example}
\newtheorem{algorithm}[thm]{Algorithm}
\newenvironment{acknowledge}{\noindent\textbf{Acknowledgments.}}{}
\numberwithin{equation}{section}
\numberwithin{thm}{section}
\long\def\blankfootnotetext#1{\begingroup\def\thefootnote{\fnsymbol{footnote}}\footnotetext{#1}\endgroup}
\newcommand{\abs}[1]{\left\vert{#1}\right\vert}
\newcommand{\set}[1]{\left\{{#1}\right\}}
\newcommand{\spn}[1]{\mathrm{span}\set{#1}}
\newcommand{\coord}[1]{\left({#1}\right)}
\newcommand{\Proj}{\mathbb{P}}
\newcommand{\Z}{\mathbb{Z}}
\newcommand{\Q}{\mathbb{Q}}
\newcommand{\R}{\mathbb{R}}
\newcommand{\conv}[1]{\mathrm{conv}\,#1}
\newcommand{\sconv}[1]{\mathrm{conv}\!\set{#1}}
\newcommand{\scone}[1]{\mathrm{cone}\!\set{#1}}
\graphicspath{{images/}}
\begin{document}
\title{Canonical toric Fano threefolds}
\author{Alexander M. Kasprzyk}
\address{Department of Mathematics and Statistics\\University of New Brunswick\\Fredericton NB\\E3B 5A3\\Canada.}
\email{kasprzyk@unb.ca}
\maketitle
\blankfootnotetext{2000 \textit{Mathematics Subject Classification}. Primary 14J45; Secondary 14J30, 14M25, 52B20.}
\blankfootnotetext{\textit{Key words and phrases}. Toric, Fano, threefold, canonical singularities, convex polytopes.}
\begin{abstract}
An inductive approach to classifying all toric Fano varieties is given. As an application of this technique, we present a classification of the toric Fano threefolds with at worst canonical singularities. Up to isomorphism, there are $674,\!688$ such varieties.
\end{abstract}
\section{Introduction}
Recall that a normal projective variety $X$ with log terminal singularities such that the anticanonical divisor $-K_X$ is an ample $\Q$-Cartier divisor is said to be \emph{Fano}. A nonsingular Fano surface is usually called a \emph{del Pezzo surface}. Their classification is well known: $\Proj^2$, $\Proj^1\times\Proj^1$, and $\Proj^2$ blown up in at most eight points (in general position). Of these, the first five are toric. Nonsingular Fano threefolds have also been classified. There are seventeen families with Picard number one, and eighty-nine other families~(\cite{Isk79a, Isk79b, MU82, Sok79, Cut89, Tak89, MM04}).

A great deal more can be said concerning nonsingular toric Fano varieties (\cite{Wis01,FS03}). There are eighteen smooth toric Fano threefolds (\cite{Bat81,Baty91,WW82}), and $124$ smooth toric Fano fourfolds (\cite{Bat99,Sat00}). An inductive algorithm for classifying the smooth toric Fano $n$-folds was recently described in~\cite{KN07}. This algorithm requires knowledge of the Gorenstein toric Fano $(n-1)$-folds and, using the data from~\cite{KS00}, allowed the classification of the five-folds. {\O}bro has presented a different algorithm based on the ingenious notion of \emph{special facets}; using this method dimensions six, seven, and eight have now been classified (see~\cite{Obr07}).

In~\cite{KMM92} it was shown that the degree $(-K_X)^n$ of any smooth Fano variety $X$ of dimension $n$ is bounded, as is the number of deformation types. Similar results are not known for Fano varieties in general; but the number of isomorphism classes of toric Fano varieties of fixed dimension and bounded discrepancy is known to be finite (see~\cite{BB92,Bor00}). It thus makes sense to look for complete classifications in the toric setting beyond the smooth cases.

Gorenstein toric Fano varieties have been classified up to dimension four. There are $16$, $4319$, and $473,\!800,\!776$  isomorphism classes in, respectively, dimension two, three, and four (see~\cite{KS97,KS98b,KS00}). These classifications are of particular interest: Gorenstein toric Fano varieties are used to construct mirror pairs of Calabi-Yau varieties (see~\cite{Bat94,BaBo96,KS02}).

One can also attempt to classify those toric Fano varieties with at worst terminal singularities. Every surface of this form is nonsingular, and so the classification reduces to the smooth case above. In three dimensions, the author showed in~\cite{Kas03} that there are (up to isomorphism) $634$ varieties, of which $233$ are $\Q$-factorial and $100$ are Gorenstein.

All the above classifications are subsets of a more general case: toric Fano varieties with at worst canonical singularities. Here the surface case reduces to the Gorenstein case. This paper describes an inductive approach to achieving a classification in higher dimensions. As an application, the classification for threefolds is calculated. There are $674,\!688$ isomorphism classes. As well as encapsulating the three--dimensional classifications mentioned above, it is worth observing that $12,\!190$ of the resulting varieties are $\Q$-factorial (of which the Picard number is bounded by $\rho\leq 7$). The classification is available online via the Graded Rings Database (\cite{Bro07}) at \href{http://malham.kent.ac.uk}{\texttt{http://malham.kent.ac.uk/}}.

The various classifications are summarised in Table~\ref{tab:classifications}.

\begin{table}[htdp]
\begin{center}
\begin{tabular}{|r||c||c|c|c||c|c|c|}
\hline
&&\multicolumn{3}{c||}{Terminal}&\multicolumn{3}{c|}{Canonical}\\
$n$&Smooth&Gorenstein&$\Q$-factorial&Total&Gorenstein&$\Q$-factorial&Total\\
\hline
2&5&5&5&5&16&16&16\\
3&18&100&233&634&4,319&12,190&674,688\\
4&124&&&&473,800,776&&\\
5&866&&&&&&\\
6&7,622&&&&&&\\
7&72,256&&&&&&\\
8&749,892&&&&&&\\
\hline
\end{tabular}
\end{center}
\caption{Known classifications of toric Fano $n$-folds.}
\label{tab:classifications}
\end{table}
\begin{acknowledge}
The author would like to express his gratitude to Dr.~G.~K.~Sankaran for his invaluable explanations and advice. A special acknowledgement is due to Professor Alexander Borisov for making~\cite{BB} available; the current paper was inspired by the ideas developed in that unpublished work. Thanks also to Dr.~Gavin Brown and the IMSAS at the University of Kent for hosting the final classification online in a searchable format, to Michael Kerber for assistance with the web interface, and to an anonymous referee for several useful observations.

A significant portion of this work was funded by an Engineering and Physical Sciences Research Council (EPSRC) studentship, and forms part of the author's PhD thesis~(\cite{KasPhD}). The author is currently funded by an ACEnet Postdoctoral Research Fellowship.
\end{acknowledge}
\section{Fano Polytopes}\label{sec:Fano_polytopes}

A \emph{toric variety} is a normal variety $X$ that contains an algebraic torus as a dense open subset, together with an action of the torus on $X$ which extends the natural action of the torus on itself. For further details see~\cite{Oda78,Dan78,Ful93}. We shall briefly review the properties we need, and in so doing fix our notation.

Let $M\cong\Z^n$ be the lattice of characters of the torus, with dual lattice $N:=\mathrm{Hom}(M,\Z)$. Every toric variety $X$ of dimension $n$ has an associated fan $\Delta$ in $N_\R:=N\otimes_\Z\R$. The converse also holds; to any fan $\Delta$ there is an associated toric variety $X(\Delta)$. Let $\set{\rho_i}_{i\in I}$ be the set of rays of $\Delta$. For each $i\in I$ there exists a unique primitive lattice element of $\rho_i$, which by a traditional abuse of notation we continue to denote $\rho_i$. $X$ is Fano if and only if $\set{\rho_i}_{i\in I}$ correspond to the vertices of a convex polytope in $N_\R$ (see, for example,~\cite{Dan78}).

A normal variety $X$ is \emph{$\Q$-factorial} if every prime divisor $\Gamma\subset X$ has a positive integer multiple
$c\Gamma$ which is a Cartier divisor. Once again, for the toric case there exists a well known description in terms of the fan. The toric variety $X$ is $\Q$-factorial if and only if the fan $\Delta$ is simplicial.

We say that a fan $\Delta$ is \emph{terminal} if each cone $\sigma\in\Delta$ satisfies the following:
\begin{enumerate}
\item The rays $\rho_1,\ldots,\rho_k$ of $\sigma$ are contained in an affine hyperplane $H:(u(v)=1)$ for some $u\in
 M_{\Q}$;
\item\label{cond:hyperplane_terminal} There are no other elements of the lattice $N$ in the part of $\sigma$ under or on $H$ (i.e. $N\cap\sigma\cap
 (u(v)\leq1)=\set{0,\rho_1,\ldots,\rho_k}$).
\end{enumerate}

A toric variety $X$ is \emph{terminal} (i.e. has at worst terminal singularities) if and only if the fan $\Delta$ is terminal. Relaxing condition~\eqref{cond:hyperplane_terminal} slightly to allow lattice points on $H$, one obtains the definition of a \emph{canonical} fan. $X$ has (at worst) canonical singularities if and only if the fan $\Delta$ is canonical~(\cite{Reid83M}).

\begin{defn}
Let $P\subset N_\R$ be a convex lattice polytope containing only the origin as a strictly interior lattice point (i.e.~$P^\circ\cap N=\set{0}$). We call such a polytope \emph{Fano}. If in addition the only boundary lattice points of $P$ are the vertices (i.e.~$\partial P\cap N=\mathrm{vert}\,{P}$) then we call $P$ a \emph{terminal Fano polytope}. Otherwise we call $P$ a \emph{canonical Fano polytope}.
\end{defn}

Clearly there is an equivalence between terminal (resp. canonical) Fano polytopes and toric Fano varieties with at worst terminal (resp. canonical) singularities. Two toric Fano $n$-folds are isomorphic if and only if the corresponding Fano polytopes are unimodular equivalent; i.e. equivalent up to a linear unimodular transformation from $GL(n,\Z)$.

In~\cite{Kas03} a classification of toric Fano threefolds with at worst terminal singularities was given. The method employed relied on an approach first outlined in~\cite{BB}. It depends on the polytopal description of a toric Fano variety, and can be summarised in two steps:
\begin{itemize}
\item[(i)] Classify all the ``minimal'' polytopes;
\item[(ii)] Inductively ``grow'' these minimal polytopes.
\end{itemize}

Let us explain this algorithm in more detail. First we shall define what we mean by minimal:

\begin{defn}\label{defn:canon_min}
Let $P$ be a canonical (resp. terminal) Fano $n$-tope. We say that $P$ is \emph{minimal} if, for all $\rho\in\mathrm{vert}\,P$, the polytope $\conv\!(P\cap N\setminus\set{\rho})$ obtained by subtracting $\rho$ from $P$ is not a canonical (resp. terminal) Fano $n$-tope.
\end{defn}

Notice that in the canonical case we are only required to check that the origin is not contained in the interior of any of the smaller polytopes obtained by subtracting a vertex. Our use of Fano and minimal will often be relative to some obvious subspace. Such occurrences should not cause any confusion. This is a common theme when considering lattice polytopes: for example, when talking about the volume of a face, one usually means the lattice volume of the face in the appropriate sublattice.

\begin{eg}\label{eg:minimal_examples}
Let $P:=\sconv{\pm e_1,\pm e_2}$, where $e_1$ and $e_2$ form a basis for $N$. $P$ is the terminal Fano polygon associated with $\Proj^1\times\Proj^1$. Let $P':=\sconv{\pm e_1}\subset P$. $P'$ is the one-dimensional terminal Fano polytope associated with $\Proj^1$. Both $P$ and $P'$ are examples of minimal Fano polytopes (in two and one dimension respectively).
\end{eg}

Given a Fano polytope $P$ one can enlarge (or ``grow'') it to $P'=\conv\!(P\cup\set{v})$ by the addition of a lattice point $v\in N$, and evaluate whether $P'$ is also a Fano polytope. Clearly, if one starts with the minimal Fano polytopes, one will achieve a complete classification using this technique.

The number of possible lattice points that can be added to $P$ to create a Fano polytope is finite. Assume that $P'$ is Fano, and consider the ray passing through the origin and $-v$. It will intersect $\partial P$ in a point $x$ on some face $F$ not containing $v$. Let $S\subset\mathrm{vert}\,P\cap F$ be of smallest size such that $x\in\conv S$; say $\abs{S}=d$, where $d\leq n$ is as small as possible. Then $\conv\!(S\cup\set{v})$ is a $d$-simplex containing the origin strictly in its (relative) interior. In other words, $\conv\!(S\cup\set{v})$ is a Fano $d$-simplex; there are finitely many of these by, for example,~\cite{BB92,Bor00}.

Thus we have an algorithm for finding all possible Fano polytopes $P'$ which can be obtained from $P$. What we require is a classification of the Fano $d$-simplices, for $d\leq n$ (actually it is sufficient to know the possible weights). Such a classification can be obtained from the techniques in~\cite{BB92} (see also~\cite{Con02,Kas08b}).

What remains to be described is a method for constructing the minimal Fano polytopes. We shall prove an inductive description of these minimal Fano polytopes in Proposition~\ref{prop:decomposition}. It shall be seen that an understanding of these minimal Fano polytopes reduces to an understanding of the Fano $d$-simplices for all $d\leq n$.

Finally, in Section~\ref{sec:min_canonical_threefolds}, we shall find all minimal canonical Fano $3$-topes. A computer can then be used to establish a complete classification of toric Fano threefolds with canonical singularities. The resulting classification is summarised in Section~\ref{sec:class_canonical_threefolds}.
\section{Decomposition of Minimal Fano Polytopes}
The results in this section should be compared with~\cite{KS97}. It should be stressed that the results ignore the lattice point structure of the Fano polytope; only the property that the Fano polytope contains the origin in its interior is relevant.

Let $x_0,\ldots,x_n\in N_\R\cong\R^n$ be such that $P:=\sconv{x_0,\ldots,x_n}$ is an $n$-simplex with $0\in P^\circ$. To this simplex we associate the complete fan $\Delta:=\Delta(P)$ given by the cones over the faces of $P$; i.e. generated by
$$\sigma_i:=\scone{x_0,\ldots,\hat{x}_i,\ldots,x_n},\qquad\text{ where }i=0,\ldots,n.$$
$\hat{x}_i$ indicates that the vertex $x_i$ is omitted.
The following lemma is immediate:
\begin{lem}\label{lem:minus_fan}
With notation as above, let $x\in N_\R$. Then $x\in(-\sigma_i)^\circ$ if and only if
$$P':=\sconv{x_0,\ldots,\hat{x}_i,\ldots,x_n,x}$$
is an $n$-simplex with $0\in{P'}^\circ$, and $\Delta(P')$ is a complete fan.
\end{lem}

We are now in a position to prove the main result of this section:
\begin{prop}\label{prop:decomposition}
Any minimal canonical (resp.~terminal) Fano $n$--tope $P$ is either a simplex, or can be written as $P=\conv\!(S\cup P')$ for some $S$ a minimal canonical (resp.~terminal) Fano $k$--simplex and $P'$ a minimal canonical (resp.~terminal) Fano $(n-k+r)$--tope, where $0\leq r<k<n$, moreover, $\dim(\mathrm\,S\cap\mathrm\,P')\leq r$, and $r$ equals the number of common vertices of $S$ and $P'$.
\end{prop}
\begin{proof}
We assume that $P$ is not a simplex. Let $x_0,\ldots,x_l$ be the vertices of $P$, where $l>n$. Without loss of generality we may assume that $x_0,\ldots,x_n$ do not lie in a hyperplane and that $0\in\sconv{x_0,\ldots,x_n}$.

Minimality of $P$ ensures that $0\notin\sconv{x_0,\ldots,x_n}^\circ$. Hence the origin must lie on some facet, and we may assume (with a possible reordering) that $0\in\sconv{x_0,\ldots,x_k}^\circ$ for some $k<n$. We obtain the $k$--simplex $S:=\sconv{x_0,\ldots,x_k}$. $S$ is minimal and Fano since $P$ is; if $P$ is terminal then $S$ must be terminal.

Let $P'':=\sconv{x_{k+1},\ldots,x_l}$, so $P=\conv\!(S\cup P'')$. Let $\Gamma$ be the $k$-dimensional subspace of $N_\R$ containing $S$. Since the $x_i$ are vertices we have that $\set{x_0,\ldots,x_k}\cap P''=\emptyset$, and since $P$ is minimal we have that $\set{x_{k+1},\ldots,x_l}\cap\Gamma=\emptyset$. It must also be that ${P''}^\circ\cap\Gamma\neq\emptyset$, otherwise $0$ would lie in a facet of $P$. Let $m:=\dim({P''}^\circ\cap\Gamma)$. A dimension count reveals that $\dim{P''}=n-k+m$.

By minimality of $P$ and Lemma~\ref{lem:minus_fan} we have that ${P''}^\circ\cap\Gamma\subset-\sigma$ for some $r$-dimensional simplicial cone $\sigma\in\Delta(S)$, where $k>r\geq m$. Since $\set{0}$ is the apex of $-\sigma$ we have that either $\set{0}={P''}^\circ\cap\Gamma$ or $0\notin{P''}^\circ\cap\Gamma$. The first case gives us that $P''$ is a minimal Fano $(n-k)$--tope (which is necessarily terminal if $P$ is terminal), so by setting $P'=P''$ we are done. For the second possibility we may assume that $\sigma=\scone{x_{k-r+1},\ldots,x_k}$ and construct the polytope $P':=\sconv{x_{k-r+1},\ldots,x_l}$. By construction $\dim P'=n-k+r$, and by Lemma~\ref{lem:minus_fan} we have that $0\in{P'}^\circ$. Hence $P'$ is our desired minimal Fano $(n-k+r)$--tope.
\end{proof}

From Proposition~\ref{prop:decomposition} we may conclude the following two corollaries, which are well-known results of Steinitz.
\begin{cor}\label{cor:max_minimal_vertices}
Any minimal Fano polytope $P$ has at most $2\dim{P}$ vertices.
\end{cor}
\begin{cor}\label{cor:max_vertices_centrally_sym}
Let $P$ be a minimal Fano polytope such that $\abs{\mathrm{vert}\,P}=2\dim{P}$. Then $P$ is centrally symmetric.
\end{cor}
For $k>1$, no $k$--simplex is centrally symmetric. Hence Corollary~\ref{cor:max_vertices_centrally_sym} is actually an ``if and only if''.

A characterisation of centrally symmetric simplicial \emph{reflexive} Fano polytopes is given in~\cite{Nill07}. These polytopes can always be embedded in the $n$-cube $\sconv{\pm e_1\pm\ldots\pm e_n}$.
\section{Minimal Canonical Fano Threefolds}\label{sec:min_canonical_threefolds}
For the convenience of the reader we begin by summarising the main results of this section in the following theorem (see also Tables~\ref{tab:minimal_tet} and~\ref{tab:minimal_canonical_dim3}):
\begin{thm}
There are $26$ minimum Fano polytopes in dimension three, up to the action of $GL(3,\Z)$. Of these sixteen are tetrahedra. 
\end{thm}

First we shall describing which of the Fano tetrahedra are minimal. We do this by restricting the possible weights which may occur.

\begin{figure}[htbp]
\centering
\subfigure[]{\includegraphics[scale=.8]{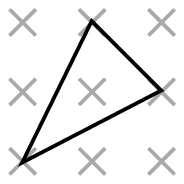}}
\hspace{.2in}
\subfigure[]{\includegraphics[scale=.8]{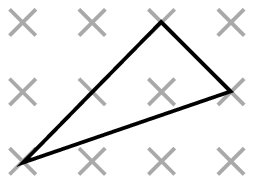}}
\hspace{.2in}
\subfigure[]{\includegraphics[scale=.8]{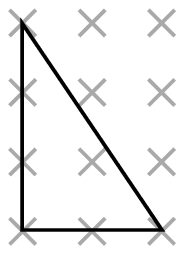}}
\hspace{.2in}
\subfigure[]{\includegraphics[scale=.8]{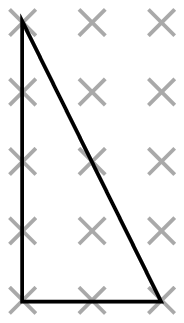}}
\hspace{.2in}
\subfigure[]{\includegraphics[scale=.8]{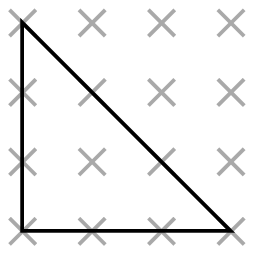}}
\caption{The five Fano triangles, with weights $(1,1,1)$, $(1,1,2)$, $(1,2,3)$, $(1,1,2)$, and $(1,1,1)$ respectively. Only (a) and (b) are minimal.}
\label{fig:Fano_triangles}
\end{figure}

\begin{defn}
Let $S$ be a Fano tetrahedron. We say that $S$ has \emph{weights} $(\lambda_0,\lambda_1,\lambda_2,\lambda_3)\in\Z_{>0}^4$ if:
$$\lambda_0x_0+\lambda_1x_2+\lambda_2x_2+\lambda_3x_3=0,$$
where the $x_i$ are the vertices of $S$, labelled in some order.
\end{defn}

Weights are unique up to reordering and scalar multiplication. It is useful to normalise them by insisting that $\lambda_0\leq\lambda_1\leq\lambda_2\leq\lambda_3$ and that $\gcd\!\set{\lambda_0,\lambda_1,\lambda_2,\lambda_3}=1$.

Before we continue, we need to be familiar with the Fano triangles. The Fano polytopes are well documented in the literature, more often than not appearing alongside an original method of proof. Consult, for example,~\cite{KS97, Sat00, PR00, Nill05}. The triangles are illustrated in Figure~\ref{fig:Fano_triangles}.

\begin{prop}\label{prop:min_canonical_tet_weights}
Let $P$ be a minimal Fano tetrahedron. The possible weights for $P$ are:
\begin{center}
\begin{tabular}{lllll}
$(1,1,1,1)$,&$(1,1,1,2)$,&$(1,1,1,3)$,&$(1,1,2,2)$,&$(1,1,2,3)$,\\
$(1,1,2,4)$,&$(1,1,3,4)$,&$(1,1,3,5)$,&$(1,1,4,6)$,&$(1,2,3,5)$,\\
$(1,3,4,5)$,&$(2,2,3,5)$,&$(2,3,5,7)$,&\multicolumn{2}{l}{or $(3,4,5,7)$.}
\end{tabular}
\end{center}
\end{prop}
\begin{proof}
If $P$ is terminal then the possible weights are listed in~\cite[Proposition 1.8]{Kas03}; they are $(1,1,1,1)$, $(1,1,1,2)$, $(1,1,2,3)$, $(1,2,3,5)$, $(1,3,4,5)$, $(2,3,5,7)$, and $(3,4,5,7)$. 

Suppose that $P=\sconv{x_0,x_1,x_2,x_3}$ is not terminal. Minimality dictates that no edge of $P$ can contain more than one interior lattice point. Let $x\in\partial P\cap N\setminus\mathrm{vert}\,P$. Since the fan $\Delta$ of $P$ is complete, so $x\in-\sigma$ for some cone $\sigma\in\Delta$ of smallest possible dimension. In particular $\dim\sigma\leq 2$, otherwise $P$ is not minimal, hence $\sigma\preccurlyeq\scone{x_0,x_1}$ without loss of generality. Because of minimality we may suppose that any non-vertex lattice point in $\sconv{x_1,x_2,x_3}$ is contained in $-\scone{x_0,x_1}$.

\noindent\underline{$\dim\sigma=1$:}

Let $x=-x_0$, where $x$ is in the interior of the face $\sconv{x_1,x_2,x_3}$, and the line segment $x_1,x$ is lattice point free. There are two possibilities: either there is a second non-vertex lattice point in the face, or there isn't.
\begin{enumerate}
\item
If $x$ is the only non-vertex lattice point in the face then we may regard $\sconv{x_1,x_2,x_3}$ as the Fano triangle (a) in Figure~\ref{fig:Fano_triangles}, with $x$ playing the role of the origin. Hence:
$$\frac{1}{3}(x_1+x_2+x_3)=-x_0,$$
and we obtain weights $(1,1,1,3)$.
\item
Suppose that there exists a second non-vertex lattice point $x'\in\conv\set{x_1,x_2,x_3}$. Then $\sconv{x',x_0,x_1}$ is a Fano triangle with $x$ on the edge joining $x_1$ and $x'$. We may choose $x'$ to be as far from $x_1$ as possible.
\begin{enumerate}
\item
$x'$ lies on the edge joining $x_2$ and $x_3$. In which case, $x'=(1/2)(x_2+x_3)$. There are only two possible Fano triangles: (b) and (c) in Figure~\ref{fig:Fano_triangles}. The former gives:
$$2x_0+x_1+\frac{1}{2}(x_2+x_3)=0,$$
and hence $P$ has weights $(1,1,2,4)$. The latter gives:
$$3x_0+2x_1+\frac{1}{2}(x_2+x_3)=0,$$
yielding weights $(1,1,4,6)$.
\item
$x'$ does not lie on the edge joining $x_2$ and $x_3$. There are no lattice points on the line segment between $x_0$ and $x'$, hence the Fano triangle $\sconv{x',x_0,x_1}$ can only be (b) (observe that (c) is impossible since there are no lattice points between $x_1$ and $x=-x_0$), and so $x'=-2x_0-x_1$. In particular, $x'$ is the only lattice point in the triangle $\sconv{x,x_2,x_3}$, hence:
$$\frac{1}{3}(x_2+x_3-x_0)=-2x_0-x_1.$$
This gives weights $(1,1,3,5)$.
\end{enumerate}
\end{enumerate}

\noindent\underline{$\dim\sigma=2$:}

We have that $\sigma=\sconv{x_0,x_1}$ and may assume that $-x_0$ and $-x_1$ are not a lattice points in the polytope, otherwise we can reduce to the previous case. Let us choose $x$ to be as far from $x_1$ as is possible. Furthermore, minimality gives that any non-vertex lattice point in $\sconv{x_1,x_2,x_3}$ must be contained in $-\scone{x_0,x_1}$.
\begin{enumerate}
\item
Suppose that $x$ lies on that edge joining $x_2$ and $x_3$. Then $x=(1/2)(x_2+x_3)$.

In this case, since the edge joining $x_0$ and $x_1$ contains at most one interior lattice point, the Fano triangle $\sconv{x,x_0,x_1}$ must be equivalent to (a), (b), or (c) from Figure~\ref{fig:Fano_triangles} (note that (d) is impossible, since $-x_0$ or $-x_1$ would be lattice points in the polytope). (a) gives equation $x+x_0+x_1=0$, yielding weights $(1,1,2,2)$. For (b) we obtain $2x+x_0+x_1=0$, giving weights $(1,1,1,1)$.

Finally we consider (c). Notice that $-x_0$ is not in the face by assumption, hence either $x+2x_0+3x_1=0$ or $2x+x_0+3x_1=0$. The second possibility gives us the lattice points $-x_1$ and $x_0+x_1+x$ on the face $\sconv{x_0,x_2,x_3}$, where the second point is closer to $x_0$ than the first. This contradicts minimality. Hence the only possibility is $(1,1,4,6)$.
\item
If $x$ does not lie on the edge joining $x_2$ and $x_3$ then $x$ is, say, in the interior of $\sconv{x_1,x_2,x_3}$, and the only possible Fano triangles for $\sconv{x,x_0,x_1}$ are (a), (b), and (c) (since the edge joining $x_0$ and $x$ must be lattice point free). (a) tells us that $x$ is the only non-vertex lattice point in the face $\sconv{x_1,x_2,x_3}$, so we obtain:
$$\frac{1}{3}(x_1+x_2+x_3)=-x_0-x_1.$$
This gives weights $(1,1,3,4)$.

Since $-x_0$ is not in the face, (b) gives us that the face has only one non-vertex lattice point. Hence:
$$\frac{1}{3}(x_1+x_2+x_3)=-\frac{1}{2}(x_0+x_1),$$
yielding weights $(2,2,3,5)$.

Possibility (c) contradicts the assumption that $-x_0$ and $-x_1$ are not in the polytope.
\end{enumerate}
\end{proof}
\begin{table}[htdp]
\centering
\begin{tabular}{|c|c|c|c|}
\hline
$(1,1,1,1)$&$(1,1,1,1)$&$(1,1,1,1)$&$(1,1,1,2)$\\
\hline
$\begin{pmatrix}-1&1&0&0\\-1&0&1&0\\-1&0&0&1\\\end{pmatrix}$&
$\begin{pmatrix}-2&2&0&0\\-2&1&1&0\\-1&0&0&1\\\end{pmatrix}$&
$\begin{pmatrix}-5&5&0&0\\-3&2&1&0\\-2&1&0&1\\\end{pmatrix}$&
$\begin{pmatrix}-1&1&0&0\\-1&0&1&0\\-2&0&0&1\\\end{pmatrix}$\\
\hline\hline
$(1,1,1,3)$&$(1,1,2,2)$&$(1,1,2,3)$&$(1,1,2,4)$\\
\hline
$\begin{pmatrix}-1&1&0&0\\-1&0&1&0\\-3&0&0&1\\\end{pmatrix}$&
$\begin{pmatrix}-1&1&0&0\\-2&0&1&0\\-2&0&0&1\\\end{pmatrix}$&
$\begin{pmatrix}-1&1&0&0\\-2&0&1&0\\-3&0&0&1\\\end{pmatrix}$&
$\begin{pmatrix}-1&1&0&0\\-2&0&1&0\\-4&0&0&1\\\end{pmatrix}$\\
\hline\hline
$(1,1,3,4)$&$(1,1,3,5)$&$(1,1,4,6)$&$(1,2,3,5)$\\
\hline
$\begin{pmatrix}-1&1&0&0\\-3&0&1&0\\-4&0&0&1\\\end{pmatrix}$&
$\begin{pmatrix}-1&1&0&0\\-3&0&1&0\\-5&0&0&1\\\end{pmatrix}$&
$\begin{pmatrix}-1&1&0&0\\-4&0&1&0\\-6&0&0&1\\\end{pmatrix}$&
$\begin{pmatrix}-2&1&0&0\\-3&0&1&0\\-5&0&0&1\\\end{pmatrix}$\\
\hline\hline
$(1,3,4,5)$&$(2,2,3,5)$&$(2,3,5,7)$&$(3,4,5,7)$\\
\hline
$\begin{pmatrix}-3&1&0&0\\-4&0&1&0\\-5&0&0&1\\\end{pmatrix}$&
$\begin{pmatrix}-1&1&0&0\\-3&0&2&0\\-4&0&1&1\\\end{pmatrix}$&
$\begin{pmatrix}-3&2&0&0\\-4&1&1&0\\-5&1&0&1\\\end{pmatrix}$&
$\begin{pmatrix}-4&3&0&0\\-3&1&1&0\\-5&2&0&1\\\end{pmatrix}$\\
\hline
\end{tabular}
\caption{The sixteen minimal canonical Fano tetrahedra.}
\label{tab:minimal_tet}
\end{table}
Knowing the weights, we can find the associated tetrahedra. We shall require the following result:
\begin{prop}[\protect{\cite[Proposition~2]{BB92}}]\label{prop:generating_fan_unique}
For any weights $(\lambda_0,\lambda_1,\ldots,\lambda_n)$ such that $\gcd\!\set{\lambda_0,\lambda_1,\ldots,\lambda_n}=1$,  let $\rho_0,\rho_1,\ldots,\rho_n\in N$ be the primitive generators for the fan of $\Proj(\lambda_0,\lambda_1,\ldots,\lambda_n)$. Then:
\begin{itemize}
\item[(i)] $\lambda_0\rho_0+\lambda_1\rho_1+\ldots+\lambda_n\rho_n=0$;
\item[(ii)] The $\rho_i$ generate the lattice $N$.
\end{itemize}

Furthermore, if $\rho'_0,\rho'_1,\ldots,\rho'_n$ is any set of primitive lattice elements satisfying (i) and (ii) then there exists a transformation in $GL(n,\Z)$ sending $\rho_i$ to $\rho'_i$ for $i=0,1,\ldots,n$.
\end{prop}
\begin{thm}
There are sixteen minimal Fano tetrahedra, whose vertices are listed (up to the action of $GL(3,\Z)$) in Table~\ref{tab:minimal_tet}.
\end{thm}
\begin{proof}
The terminal Fano tetrahedra are listed in~\cite[Table~4]{Kas03}. We need only consider the canonical cases.

From the proof of Proposition~\ref{prop:min_canonical_tet_weights} we can see when the vertices of a minimal tetrahedron generate the lattice $N$. When this is the case, Proposition~\ref{prop:generating_fan_unique} tells us that the tetrahedron corresponds to weighted projective space. This is the only possibility for all weights except $(1,1,1,1)$ (in the notation of the proof, we are considering $\dim\sigma=2$, case~(1)(b)). This gives a tetrahedron whose vertices generate an index two sublattice. This corresponds to a fake weighted projective space of index two;~\cite{Con02} describes how to compute the vertices of the tetrahedron.

\end{proof}
It should be emphasised that not every Fano tetrahedron is minimal. As mentioned in~\cite[pg.~278]{BB92}, there are a total of $225$ Fano tetrahedra; see the appendix of~\cite{BB} for the complete list. This has been verified by the author using the bounds described in~\cite{Kas08b}. There are $104$ distinct weights, which are listed in Table~\ref{tab:all_tet_weights}.

\begin{table}[htdp]
\centering
\begin{tabular}{|c|c|}
\hline
Weights&Sum\\
\hline
$(1,1,1,1)$&$4$\\
$(1,1,1,2)$&$5$\\
$(1,1,1,3)$&$6$\\
$(1,1,2,2)$&$6$\\
$(1,1,2,3)$&$7$\\
$(1,1,2,4)$&$8$\\
$(1,2,2,3)$&$8$\\
$(1,1,3,4)$&$9$\\
$(1,2,3,3)$&$9$\\
$(1,1,3,5)$&$10$\\
$(1,2,2,5)$&$10$\\
$(1,2,3,4)$&$10$\\
$(1,2,3,5)$&$11$\\
$(1,1,4,6)$&$12$\\
$(1,2,3,6)$&$12$\\
$(1,2,4,5)$&$12$\\
$(1,3,4,4)$&$12$\\
$(2,2,3,5)$&$12$\\
$(2,3,3,4)$&$12$\\
$(1,3,4,5)$&$13$\\
$(1,2,4,7)$&$14$\\
$(2,2,3,7)$&$14$\\
$(2,3,4,5)$&$14$\\
$(1,2,5,7)$&$15$\\
$(1,3,4,7)$&$15$\\
$(1,3,5,6)$&$15$\\
\hline
\end{tabular}
\begin{tabular}{|c|c|}
\hline
Weights&Sum\\
\hline
$(2,3,5,5)$&$15$\\
$(3,3,4,5)$&$15$\\
$(1,2,5,8)$&$16$\\
$(1,3,4,8)$&$16$\\
$(1,4,5,6)$&$16$\\
$(2,3,4,7)$&$16$\\
$(2,3,5,7)$&$17$\\
$(1,2,6,9)$&$18$\\
$(1,3,5,9)$&$18$\\
$(1,4,6,7)$&$18$\\
$(2,3,4,9)$&$18$\\
$(2,3,5,8)$&$18$\\
$(3,4,5,6)$&$18$\\
$(3,4,5,7)$&$19$\\
$(1,4,5,10)$&$20$\\
$(1,5,6,8)$&$20$\\
$(2,3,5,10)$&$20$\\
$(2,4,5,9)$&$20$\\
$(2,5,6,7)$&$20$\\
$(3,4,5,8)$&$20$\\
$(1,3,7,10)$&$21$\\
$(1,4,7,9)$&$21$\\
$(1,5,7,8)$&$21$\\
$(2,3,7,9)$&$21$\\
$(3,5,6,7)$&$21$\\
$(1,3,7,11)$&$22$\\
\hline
\end{tabular}
\begin{tabular}{|c|c|}
\hline
Weights&Sum\\
\hline
$(1,4,6,11)$&$22$\\
$(2,4,5,11)$&$22$\\
$(1,3,8,12)$&$24$\\
$(1,6,8,9)$&$24$\\
$(2,3,7,12)$&$24$\\
$(2,3,8,11)$&$24$\\
$(2,5,8,9)$&$24$\\
$(3,4,5,12)$&$24$\\
$(3,4,7,10)$&$24$\\
$(3,6,7,8)$&$24$\\
$(4,5,6,9)$&$24$\\
$(4,5,7,9)$&$25$\\
$(1,5,7,13)$&$26$\\
$(2,3,8,13)$&$26$\\
$(2,5,6,13)$&$26$\\
$(2,5,9,11)$&$27$\\
$(5,6,7,9)$&$27$\\
$(1,4,9,14)$&$28$\\
$(1,5,8,14)$&$28$\\
$(3,4,7,14)$&$28$\\
$(3,7,8,10)$&$28$\\
$(4,6,7,11)$&$28$\\
$(1,4,10,15)$&$30$\\
$(1,6,8,15)$&$30$\\
$(2,3,10,15)$&$30$\\
$(2,6,7,15)$&$30$\\
\hline
\end{tabular}
\begin{tabular}{|c|c|}
\hline
Weights&Sum\\
\hline
$(3,4,10,13)$&$30$\\
$(4,5,6,15)$&$30$\\
$(4,7,9,10)$&$30$\\
$(5,6,8,11)$&$30$\\
$(2,5,9,16)$&$32$\\
$(4,5,7,16)$&$32$\\
$(3,5,11,14)$&$33$\\
$(5,8,9,11)$&$33$\\
$(3,4,10,17)$&$34$\\
$(4,6,7,17)$&$34$\\
$(1,5,12,18)$&$36$\\
$(3,4,11,18)$&$36$\\
$(3,7,8,18)$&$36$\\
$(7,8,9,12)$&$36$\\
$(3,5,11,19)$&$38$\\
$(5,6,8,19)$&$38$\\
$(5,7,8,20)$&$40$\\
$(1,6,14,21)$&$42$\\
$(2,5,14,21)$&$42$\\
$(3,4,14,21)$&$42$\\
$(4,5,13,22)$&$44$\\
$(5,8,9,22)$&$44$\\
$(3,5,16,24)$&$48$\\
$(7,8,10,25)$&$50$\\
$(4,5,18,27)$&$54$\\
$(5,6,22,33)$&$66$\\
\hline
\end{tabular}
\caption{The $104$ distinct weights occuring for the $225$ Fano tetrahedra.}
\label{tab:all_tet_weights}
\end{table}

Proposition~\ref{prop:decomposition} allows us to calculate the non-simplex minimal Fano $3$-topes. Assume we have chosen $S$ and $P'$ such that $k$ is as small as possible. If $k=1$ then $r=0$ and we have that $S$ is the polytope for $\Proj^1$, and $P'$ is a minimal Fano polygon (the minimal Fano polygons are the triangles (a) and (b) in Figure~\ref{fig:Fano_triangles} and the polygon associated with $\Proj^1\times\Proj^1$ mentioned in Example~\ref{eg:minimal_examples}). These possibilities are classified in Lemmas~\ref{lem:3and3}--\ref{lem:3and1}. The alternative is that $k=2$. Since the polygon for $\Proj^1\times \Proj^1$ contains the polytope for $\Proj^1$, it can be excluded; we need only consider the cases when $r=1$ and $P'$ is a minimal Fano triangle. Hence the Fano polytope has five vertices. These cases will be classified in Lemmas~\ref{lem:2and2}--\ref{lem:1and1}. We find that there are exactly ten non-simplex minimal Fano polytopes in dimension three. The results are collated in Table~\ref{tab:minimal_canonical_dim3}.

Once the minimal polytopes are known, the following result is immediate\footnote{My thanks to Professor Victor Batyrev for this observation.}:
\begin{thm}\label{thm:canonical_Fano_degree_bound}
Let $X$ be a toric Fano threefold with at worst canonical singularities. Then $(-K_X)^3\leq 72$. If $(-K_X)^3=72$ then $X$ is isomorphic to $\Proj(1,1,1,3)$ or $\Proj(1,1,4,6)$.
\end{thm}
\begin{proof}
Let $P_X$ be the polytope associated with $X$. There exists a minimal polytope $Q$ such that $Q\subset P_X$, hence $P_X^\vee\subset Q^\vee$. Inspection gives $\mathrm{vol}\,Q^\vee\leq 12$, hence $(-K_X)^3\leq 3!\cdot 12$.
\end{proof}
Theorem~\ref{thm:canonical_Fano_degree_bound} should be compared with the following result, conjectured by Fano and Iskovskikh and proved by Prokhorov:
\begin{thm}[\cite{Pro05}]
Let $X$ be a Gorenstein Fano threefold with at worst canonical singularities. Then $(-K_X)^3\leq 72$. If $(-K_X)^3=72$ then $X$ is isomorphic to $\Proj(1,1,1,3)$ or $\Proj(1,1,4,6)$.
\end{thm}

For the following two results minimality ensures that any such Fano polytope must be at worst terminal; these were classified in~\cite[Lemma~3.4 and~3.5]{Kas03}.
\begin{lem}\label{lem:3and3}
The minimal Fano polytopes obtained from adding the points $\pm x$ to a Fano square are equivalent to:
$$\begin{pmatrix}1&0&0&-1&0&0\\0&1&0&0&-1&0\\0&0&1&0&0&-1\end{pmatrix}\text{ or }
\begin{pmatrix}1&0&-1&0&1&-1\\0&1&0&-1&1&-1\\0&0&0&0&2&-2\end{pmatrix}.$$
\end{lem}
\begin{lem}\label{lem:3and2}
The minimal Fano polytopes containing a Fano triangle equivalent to Figure~\ref{fig:Fano_triangles}~(a), along with a pair of points $\pm x$ not lying in the plane containing the Fano triangle, are equivalent to:
$$\begin{pmatrix}1&0&0&0&-1\\0&1&0&0&-1\\0&0&1&-1&0\end{pmatrix}\text{ or }
\begin{pmatrix}1&0&-1&1&-1\\0&1&-1&2&-2\\0&0&0&3&-3\end{pmatrix}.$$
\end{lem}
\begin{lem}\label{lem:3and1}
Any minimal Fano polytope containing the minimal Fano triangle shown in Figure~\ref{fig:Fano_triangles}~(b), along with a pair of points $\pm x$ not lying in the same subspace as the triangle, is equivalent to one of:
$$\begin{pmatrix}1&0&0&0&-2\\0&1&0&0&-1\\0&0&1&-1&0\end{pmatrix}\text{ or }
\begin{pmatrix}1&0&-2&1&-1\\0&1&-1&1&-1\\0&0&0&2&-2\end{pmatrix}.$$
\end{lem}
\begin{proof}
Arrange matters such that $P:=\sconv{e_1,e_2,-2e_1-e_2,x,-x}$; $x:=\coord{a,b,c}$ is such that $0\leq a,b<c$. Clearly $a=0,b=0,c=1$ is a solution. Let us assume that $c>1$.

Since $x\neq e_3$ we cannot have $e_3\in P$, since then removing $x$ would yield a smaller canonical Fano polytope with vertex $e_3$, contradicting minimality.

Hence $e_3\notin P$ and consider the line connecting $e_3$ to the origin. If $a\geq 2b$ this line intersects $\sconv{-e_1,-2e_1-e_2,x}$ at the point $ke_3$, where $k=c/(a-b+1)$. This tells us that $k<1$, thus $a-b\geq c$, which contradicts our assumptions.

It must be that $a<2b$. The line joining $e_3$ and $0$ intersects $\sconv{e_1,-2e_1-e_2,x}$ at the point $ke_3$, where $k=c/(3b-a+1)$. Hence:
\begin{equation}\label{lem:3and1Eq1}
3b-a\geq c.
\end{equation}

As before $-e_3\notin P$.  The line joining the origin and $-e_3$ intersects $\sconv{e_1,e_2,-x}$ at the point $k(-e_3)$, where $k=c/(a+b+1)$. Thus we obtain:
\begin{equation}\label{lem:3and1Eq2}
a+b\geq c.
\end{equation}

As before $-e_1-e_3\notin P$. The line connecting the origin with $-e_1-e_3$ intersects $\sconv{-e_1,e_2,-x}$ at the point $k(-e_1-e_3)$, where $k=c/(c+b-a+1)$. Hence:
\begin{equation}\label{lem:3and1Eq3}
b\geq a.
\end{equation}

Finally, let us consider the point $-e_1-e_2-e_3$. This point must lie outside $P$, for otherwise $\sconv{e_1,e_2,-e_1-e_2-e_3,x}$ would be a Fano tetrahedron. We consider the line connecting $0$ and this point. If $2b-a>c$ then the line intersects $\sconv{-e_1,-2e_1-e_2,-x}$ at the point $k(-e_1-e_2-e_3)$, where $k=c/(b-a+1)$. But this yields $b-a\geq c$, a contradiction. Hence it must be that $2b-a\leq c$, and the line intersects $\sconv{e_1,-2e_1-e_2,-x}$. This occurs when $k=c/(a-3b+2c+1)$, and gives us:
\begin{equation}\label{lem:3and1Eq4}
c\geq 3b-a.
\end{equation}

Combining equations~\eqref{lem:3and1Eq1} and~\eqref{lem:3and1Eq4} tells us that $c=3b-a$, and by applying equation~\eqref{lem:3and1Eq2} we see that $a\geq b$. Of course equation~\eqref{lem:3and1Eq3} now tells us that $a=b$, and so $x=\coord{a,a,2a}$. This forces $a=1$.
\end{proof}
\begin{lem}\label{lem:2and2}
The minimal Fano polytopes containing two copies of the Fano triangle shown in Figure~\ref{fig:Fano_triangles}~(a) are equivalent to:
$$\begin{pmatrix}1&0&0&-1&1\\0&1&0&-1&1\\0&0&-1&0&1\end{pmatrix}.$$
\end{lem}
\begin{proof}
Let us fix the lattice such that $P:=\sconv{e_1,e_2,-e_1-e_2,x,y}$, where $x:=\sconv{a+1,b+1,c}, y:=\sconv{-a,-b,-c},$ and $0<a+1\leq b+1\leq c$. Clearly $a=0,b=0,c=1$ is a solution. Assume that $c>1$.

By minimality $-e_3$ lies outside $P$. The line connecting $-e_3$ with the origin intersects $\sconv{e_1,e_2,y}$ at the point $-ke_3$, where $k=c/(a+b+1)$. We see that:
\begin{equation}\label{lem:2and2Eq1}
c\leq a+b.
\end{equation}

Consider the point $e_1+e_2+e_3$. The line joining this point and the origin intersects $\sconv{e_1,e_2,x}$ at $k(e_1+e_2+e_3)$, where $k=c/(2c-(a+1)-(b+1)+1).$ If $e_1+e_2+e_3\notin P$ then $(a+1)+(b+1)\leq c$, contradicting equation~\eqref{lem:2and2Eq1}. Hence $e_1+e_2+e_3$ lies on the boundary of $P$, and $(a+1)+(b+1)-1=c$. But again we find that this contradicts equation~\eqref{lem:2and2Eq1}.
\end{proof}
\begin{lem}\label{lem:2and1}
Any minimal Fano polytope containing one copy of each of the two minimal Fano triangles (Figure~\ref{fig:Fano_triangles}~(a) and~(b)) is equivalent to:
$$\begin{pmatrix}1&0&0&-2&-1\\0&1&0&-1&0\\0&0&1&0&-1\end{pmatrix}.$$
\end{lem}
\begin{proof}
Arrange matters so that $P:=\sconv{e_1,e_2,-2e_1-e_2,x,y}$. There are two cases to consider:
\begin{enumerate}
\item[(i)] $x+y+e_2=0$;
\item[(ii)] $x+y+e_1=0$.
\end{enumerate}

Observe that in case (i), the line joining $e_1$ and $-2e_1-e_2$ intersects $\spn{e_2}$ at the point $-(1/3)e_2$, whereas the line joining $x$ and $y$ intersects $\spn{e_2}$ at $-(1/2)e_2$. Hence $P\setminus\set{-2e_1-e_2}$ is still Fano, which contradicts minimality of $P$. Indeed, this case reduces to those polytopes discussed in Lemma~\ref{lem:3and2}.

We now address case (ii).

We have that $x=\coord{a,b,c},y=\coord{-a-1,-b,-c},$ and can insist that $0\leq a,b<c$. Clearly $a=0,b=0,c=1$ is a solution, so suppose that $c>1$. By minimality $e_3\notin P$.

Note that the point $-e_1$ lies on the line joining $e_2$ and $-2e_1-e_2$, whilst the line joining $x$ to $y$ intersects the plane $\spn{e_1,e_2}$ at $-(1/2)e_1$. Hence this line (without the end points) is contained strictly in the interior of $P$.

The point $e_1+e_2+e_3$ lies outside $P$, otherwise $\sconv{e_1, -2e_1-e_2, e_1+e_2+e_3, y}$ is a Fano tetrahedron contained in $P$. The line connecting this point to $0$ must intersect $\sconv{e_1,e_2,x}$. This occurs at $k(e_1+e_2+e_3)$, where $k=c/(2c-a-b+1)$. We thus have:
\begin{equation}\label{lem:2and1Eq3}
a+b\leq c.
\end{equation}

The point $-e_1-e_3$ must lie outside $P$, otherwise $P$ contains the Fano tetrahedron $\sconv{e_1,-2e_1-e_2,-e_1-e_3,x}$, contradicting minimality of $P$. The line originating at $0$ and passing through $-e_1-e_3$ intersects $\partial P$ in either $\sconv{e_1,e_2,y}$ or $\sconv{-e_1,e_2,y}$. The first possibility gives the point of intersection to be $k(-e_1-e_3)$, where $k=c/(a+b-c+2)$, and we have that $a+b+1\geq 2c$. Combining this with equation~\eqref{lem:2and1Eq3} yields a contradiction.

Consider the second possibility; the line connecting $-e_1-e_3$ and the origin intersects $\sconv{-e_1,e_2,y}$ at the point $k(-e_1-e_3)$ where $k=c/(c+b-a)$. We have that:
\begin{equation}\label{lem:2and1Eq7}
b\geq a+1.
\end{equation}

Finally, consider the point $e_2+e_3$. This point must lie outside $P$; if $e_2+e_3$ were contained in $P$, then $\sconv{e_1,e_2+e_3,-2e_1-e_2,y}$ would be a Fano tetrahedron. The line joining the point with the origin intersects $\sconv{e_1,-2e_1-e_2,x}$ or $\sconv{-e_1,e_2,x}$. In the first case the point of intersection is given by $k(e_2+e_3)$, where $k=c/(c-a-b+1)$. Hence  $a+b\leq 0$, which is an impossibility (since $c\neq 1$).

The alternative is that the line intersects $\sconv{-e_1,e_2,x}$. This occurs at the point $k(e_2+e_3)$, where $k=c/(a-b+c+1)$, and we see that $a\geq b$. By considering equation~\eqref{lem:2and1Eq7} we obtain our final contradiction.
\end{proof}
\begin{table}[tdp]
\centering
\begin{tabular}[t]{|c|c|}
\hline
Comments&Vertices\\\hline
\begin{tabular}{c}$5$ Vertices\\Simplicial\\Terminal\end{tabular}&$\begin{pmatrix}1&0&0&0&-1\\0&1&0&0&-1\\0&0&1&-1&0\end{pmatrix}$\\\hline
\begin{tabular}{c}$5$ Vertices\\Simplicial\\Terminal\end{tabular}&$\begin{pmatrix}1&0&-1&1&-1\\0&1&-1&2&-2\\0&0&0&3&-3\end{pmatrix}$\\\hline
\begin{tabular}{c}$5$ Vertices\\Simplicial\end{tabular}&$\begin{pmatrix}1&0&0&0&-2\\0&1&0&0&-1\\0&0&1&-1&0\end{pmatrix}$\\\hline
\begin{tabular}{c}$5$ Vertices\\Simplicial\end{tabular}&$\begin{pmatrix}1&0&-2&1&-1\\0&1&-1&1&-1\\0&0&0&2&-2\end{pmatrix}$\\\hline
\begin{tabular}{c}$5$ Vertices\\Terminal\end{tabular}&$\begin{pmatrix}1&0&0&-1&1\\0&1&0&-1&1\\0&0&1&0&-1\end{pmatrix}$\\\hline
\end{tabular}
\begin{tabular}[t]{|c|c|}
\hline
Comments&Vertices\\\hline
\begin{tabular}{c}$5$ Vertices\\Simplicial\end{tabular}&$\begin{pmatrix}1&0&0&-2&-1\\0&1&0&-1&0\\0&0&1&0&-1\end{pmatrix}$\\\hline
$5$ Vertices&$\begin{pmatrix}1&0&0&-2&-2\\0&1&0&-1&0\\0&0&1&0&-1\end{pmatrix}$\\\hline
$5$ Vertices&$\begin{pmatrix}1&0&-2&1&-3\\0&1&-1&1&-1\\0&0&0&2&-2\end{pmatrix}$\\\hline
\begin{tabular}{c}$6$ Vertices\\Simplicial\\Terminal\end{tabular}&$\begin{pmatrix}1&0&0&-1&0&0\\0&1&0&0&-1&0\\0&0&1&0&0&-1\end{pmatrix}$\\\hline
\begin{tabular}{c}$6$ Vertices\\Simplicial\\Terminal\end{tabular}&$\begin{pmatrix}1&0&-1&0&1&-1\\0&1&0&-1&1&-1\\0&0&0&0&2&-2\end{pmatrix}$\\\hline
\end{tabular}
\caption{The non-simplex three-dimensional minimal Fano polytopes.}
\label{tab:minimal_canonical_dim3}
\end{table}
\begin{lem}\label{lem:1and1}
Any minimal Fano polytope containing two copies of the minimal Fano triangle of type $\Proj(1,1,2)$ is equivalent to:
$$\begin{pmatrix}1&0&0&-2&-2\\0&1&0&-1&0\\0&0&1&0&-1\end{pmatrix}\text{ or }
\begin{pmatrix}1&0&-2&1&-3\\0&1&-1&1&-1\\0&0&0&2&-2\end{pmatrix}.$$
\end{lem}
\begin{proof}
Fix the lattice such that $P:=\sconv{e_1,e_2,-2e_1-e_2,x,y}$. Again there are two cases to consider. If $x+y+2e_2=0$ then $-e_2$ is contained on the boundary of $P$. We already know that $-e_1$ lies on the boundary of $P$, and hence minimality reduced us to the case considered in Lemma~\ref{lem:3and1}. Thus $x+y+2e_1=0$ and $x=\coord{a,b,c}, y=\coord{-a-2,-b,-c}$, where $0\leq a,b<0$. Clearly $a=0,b=0,c=1$ is a solution. Let us assume that $c>1$.

By minimality $-e_3\notin P$. The line joining $-e_3$ to the origin intersects $\sconv{e_1,e_2,y}$ at the point $k(-e_3)$, where $k=c/(a+b+3)$. Hence we conclude that:
\begin{equation}\label{lem:1and1Eq1}
a+b+2\geq c.
\end{equation}

The point $e_1+e_2+e_3$ does not lie in $P$, otherwise either:
$$\sconv{e_1,-2e_1-e_2,e_1+e_1+e_3,y},$$
or:
$$\sconv{-e_1,-2e_1-e_2,e_1+e_2+e_3,y},$$
would be a Fano tetrahedron. Consider the line connecting $0$ and $e_1+e_2+e_3$. This line intersects $\sconv{e_1,e_2,x}$ at the point $k(e_1+e_2+e_3)$, where $k=c/(2c-a-b+1)$. In particular,
\begin{equation}\label{lem:1and1Eq2}
a+b\leq c.
\end{equation}

If $e_2+e_3\in P$ then $\sconv{e_1,-2e_1-e_2,e_2+e_3,y}$ would be a Fano tetrahedron. This is not permissible. The line connecting $e_2+e_3$ and the origin intersects $\sconv{-e_1,e_2,x}$ at the point $k(e_2+e_3)$, where $k=c/(a-b+c+1)$. We conclude that:
\begin{equation}\label{lem:1and1Eq3}
a\geq b.
\end{equation}
In particular $a\neq 0$, since the alternative would force $c=1$.

Finally we consider the point $-e_1-e_3$. The line connecting this point with the origin intersects $\sconv{-e_1,e_2,y}$ if $a+2\leq c$, or $\sconv{e_1,e_2,y}$ if $a+2>c$. The first possibility gives the point of intersection as $k(-e_1-e_3)$, where $k=c/(b+c-a-1)$. If $-e_1-e_3$ lies on the boundary of $P$, we see that $b=a+1$. This contradicts equation~\eqref{lem:1and1Eq3}. Hence it must be that $-e_1-e_3$ lies outside $P$. In this case, $b\geq a+2$, and once again this contradicts equation~\eqref{lem:1and1Eq3}. It must be that $a+2>c$, which implies that $a=c-1$. Equation~\eqref{lem:1and1Eq2} forces $b\leq 1$, and by applying equation~\eqref{lem:1and1Eq1} we see that the only possibility is $a=1,b=1,c=2$.
\end{proof}
\section{Canonical Toric Fano Threefolds}\label{sec:class_canonical_threefolds}
Using the results of Section~\ref{sec:min_canonical_threefolds} a computer classification of all canonical Fano polytopes of dimension three is possible. This is a significant undertaking; a month of computation on a parallel computing system was required. The code, written in C, is available from the author upon request. It should be emphasised that several known results exist as sub-classifications, and that the resulting list can be independently checked using packages such as PALP~\cite{KS04}. We summarise the algorithm below.

\algorithm\label{alg:classification}
For each of the $26$ minimal Fano polytopes given in Tables~\ref{tab:minimal_tet} and~\ref{tab:minimal_canonical_dim3}, perform the following recursive algorithm:
\begin{enumerate}
\item\label{algstep:build_list}\emph{Identifying unimodular equivalence:}
We have been given a canonical Fano polytope $P$, and inductively are constructing a set $\mathcal{P}$ which will ultimately contain all possible canonical Fano polytopes, up to unimodular equivalent. Thus for each $Q\in\mathcal{P}$, check whether there exists a transformation in $GL(3,\Z)$ sending the vertices of $P$ bijectively onto the vertices of $Q$. If $P$ is new then add it to $\mathcal{P}$ and proceed to step~\eqref{algstep:vertices}. Obviously invariants of the two polytopes such as their volume, degree, whether they are both simplicial, etc.\ can be used to greatly reduce the number of comparisons required.
\item\label{algstep:vertices}\emph{Successively choosing new vertices:}
We have been given a canonical Fano polytope $P$ and wish to extend $P$ via the addition of a new vertex.
\begin{enumerate}
\item For each vertex $v$ of $P$ such that $P':=\conv(P\cup\set{-v})$ is a canonical Fano polytope with $-v\in\mathrm{vert}\,P'$, recuse on step~\eqref{algstep:build_list} with $P'$.
\item For each pair of distinct vertices $v_1$ and $v_2$ check which of the following six sums give a lattice point $v\in N$ (cf.~Figure~\ref{fig:Fano_triangles}):
$$
\begin{array}{lll}
-v_1-v_2,\\
-2v_1-v_2,&-\frac{1}{2}v_1-\frac{1}{2}v_2,\\
-2v_1-3v_2,&-\frac{1}{2}v_1-\frac{3}{2}v_2,&-\frac{1}{3}v_1-\frac{2}{3}v_2.
\end{array}
$$
In each case, if $P':=\conv(P\cup\set{v})$ is a canonical Fano polytope with $v\in\mathrm{vert}\,P'$, then recuse on step~\eqref{algstep:build_list} with $P'$.
\item For each choice of pair-wise distinct vertices $v_1,v_2$, and $v_3$, and for each weight $(\lambda_0,\lambda_1,\lambda_2,\lambda_3)$ in Table~\ref{tab:all_tet_weights}, check whether any of the four sums:
$$
\begin{array}{ll}
-\frac{\lambda_1}{\lambda_0}v_1-\frac{\lambda_2}{\lambda_0}v_2-\frac{\lambda_3}{\lambda_0}v_3,&
-\frac{\lambda_0}{\lambda_1}v_1-\frac{\lambda_2}{\lambda_1}v_2-\frac{\lambda_3}{\lambda_1}v_3,\\
-\frac{\lambda_1}{\lambda_2}v_1-\frac{\lambda_0}{\lambda_2}v_2-\frac{\lambda_3}{\lambda_2}v_3,&
-\frac{\lambda_1}{\lambda_3}v_1-\frac{\lambda_2}{\lambda_3}v_2-\frac{\lambda_0}{\lambda_3}v_3,
\end{array}
$$
give a lattice point $v\in N$. In each case, if $P':=\conv(P\cup\set{v})$ is a canonical Fano polytope with $v\in\mathrm{vert}\,P'$, then recuse on step~\eqref{algstep:build_list} with $P'$.
\end{enumerate}
\end{enumerate}

The final classification is available online, in a searchable format, via the Graded Rings Database at \href{http://malham.kent.ac.uk}{\texttt{http://malham.kent.ac.uk/}}. The key results are summarised below; for further details consult the online database.

\begin{thm}
Up to isomorphism, there exist exactly $674,\!688$ toric Fano threefolds. Of these, $18$ are smooth, $634$ have at worst terminal singularities, $4,\!319$ are Gorenstein, and $12,\!190$ are $\Q$-factorial. Amongst the $\Q$-factorial varieties, the rank of the Picard group is bounded by $\rho\leq 7$; this bound is attained in exactly two cases -- once when the variety is terminal, once when the variety is canonical.
\end{thm}
\bibliographystyle{amsalpha}
\providecommand{\bysame}{\leavevmode\hbox to3em{\hrulefill}\thinspace}
\providecommand{\MR}{\relax\ifhmode\unskip\space\fi MR }
\providecommand{\MRhref}[2]{\href{http://www.ams.org/mathscinet-getitem?mr=#1}{#2}}
\providecommand{\href}[2]{#2}

\end{document}